\documentclass[12pt]{amsart}

\usepackage{amssymb}
\usepackage{bbm}
\usepackage{amsmath}
\usepackage{amsthm}
\usepackage{enumitem}
\usepackage{amsfonts}
\usepackage{hyperref}
\usepackage{comment}
\usepackage{xcolor}
\usepackage{cancel} 

\usepackage{graphicx}

\usepackage{layout}

\makeatletter
\@namedef{subjclassname@2020}{%
  \textup{2020} Mathematics Subject Classification}
\makeatother

\usepackage[T1]{fontenc}

\newtheorem{thm}{Theorem}[section]
\newtheorem{cor}[thm]{Corollary}
\newtheorem{lem}[thm]{Lemma}
\newtheorem{prop}[thm]{Proposition}

\newtheorem*{mthm}{Main Theorem}

\theoremstyle{definition}
\newtheorem{dfn}[thm]{Definition}
\newtheorem{rem}[thm]{Remark}
\newtheorem{notation}[thm]{Notation}
\newtheorem{fact}[thm]{Fact}
\newtheorem{axiom}{Axiom}
\newtheorem{axiomsc}[axiom]{Axiom scheme}
\newtheorem*{convention}{Convention}

\numberwithin{equation}{section}

\DeclareMathOperator{\cMe}{c}

\newcommand{\fp}[1]{\left[#1\right]}

\newcommand{\ip}[1]{\left\lfloor#1\right\rfloor}
\newcommand{\NN}{\mathbbm{N}}
\newcommand{\ZZ}{\mathbbm{Z}}
\newcommand{\QQ}{\mathbbm{Q}}
\newcommand{\QQc}{\mathbbm{Q}^{\cMe}}
\newcommand{\RR}{\mathbbm{R}}
\newcommand{\Za}[1][\alpha]{\mathcal{Z}_{#1}}
\newcommand{\Ta}{\mathcal{T}_\alpha}
\newcommand{\Tnalg}{\mathcal{T}_{\nalg}}

\DeclareMathOperator{\nalg}{nalg}
\DeclareMathOperator{\Th}{Th}

\begin{document}


\title[Additive integers with a function for a Beatty sequence]{Model-completeness and decidability of the additive structure of integers expanded with a function for a Beatty sequence}

\author[M. Khani]{Mohsen Khani}
\address{Department of Mathematical Sciences\\ Isfahan University of Technology\\ Isfahan \\84156-83111, Iran}
\email{mohsen.khani@iut.ac.ir}
\author[A. N. Valizadeh]{Ali N. Valizadeh}
\address{School of Mathematics\\ Institute for Research in Fundamental Sciences (IPM)\\Tehran \\
	19395-5746, Iran  }
\email{valizadeh.ali@ipm.ir}
\author[A. Zarei]{Afshin Zarei}
\address{Department of Mathematical Sciences\\ Isfahan University of Technology\\ Isfahan \\84156-83111, Iran}
\email{afshin.zarei@math.iut.ac.ir}

\date{\today}

\begin{abstract}
		
			We introduce a  model-complete  theory which completely axiomatizes the structure $\Za=\langle \ZZ,+,0,1,f\rangle$ where 
		$f:x\mapsto \lfloor \alpha x\rfloor$ is a unary function with $\alpha$ a fixed transcendental number.
When $\alpha$ is computable, our theory is recursively enumerable, and hence decidable as a result of completeness.
		Therefore, this result fits into the more general theme of adding traces of multiplication
		to integers
		without losing decidability. 

	\end{abstract}

\subjclass[2020]{Primary 03B25; Secondary 03C10, 11U09, 11U05}

\keywords{Additive Group of Integers, Beatty sequence, Kronecker's lemma, Decidability, Model-completeness}

\maketitle

	\section{Introduction}
The subject of this study is the structure $\Za=\langle \ZZ,+,0,1,f\rangle$ which contains the integer addition together with a \textit{trace} of multiplication; namely the function $\ip{\alpha x}$ whose range is the Beatty sequence with modulo $\alpha$.

\par 
Our results lie in the intersection of two active research programs.
On the one hand, it relates to the recent works on  
decidability of the expansions of $\langle \ZZ, +\rangle$ as well as their classification either as stable structures, as in
\cite{conant} and \cite{conant-pillay}, or
unstable structures as in \cite{kaplan}.
In this sense, $\Za$ is   an instance of an unstable yet decidable expansion of $\langle \mathbb{Z},+\rangle$.
\par 
On the other hand, $ \Za $ is definable in
the structure $ \mathcal{R}_\alpha=\langle \RR,<,+,0,\ZZ,\alpha\ZZ\rangle$ which lies 
in the more general theme of research studying the expansions of 
real line with specific discrete additive subgroups.  
	Most relevant to our work is Hieronymi's theorem in \cite{Hieronymi-ExpansionByTwoDiscrete} which shows, for the special case
	of a quadratic $ \alpha $, that the theory of the structure $ \mathcal{R}_\alpha$---and as a result
$\Za$---is decidable. Decidability is proved there by showing that $\mathcal{R}_\alpha$	is definable in the monadic second-order structure $\langle \NN,P(\NN),\in , x\mapsto x+1\rangle$ which was already known to be decidable (\cite{Buchi-OnDecidable}).

\par 	
The results on $\mathcal{Z}_\alpha$ have been recently generalized in several directions. In particular, it is shown in \cite{Hieronymi-SturmianWords} that the common theory of the structures $ \mathcal{Z}_\alpha $ is decidable when $ \alpha $ ranges over irrational numbers less than one. A main feature of the proof appeared there is designing a B\"{u}chi automaton which can perform addition over Ostrowski numeration systems.

\par 

The current study is  along similar lines to
the result by
the first and  third authors in \cite{af-moh}
where they applied only simple techniques of model theory
to provide
an alternative proof for the
decidability of $\Za$ in the case that $ \alpha $ is the golden ratio.
Here, by applying elementary tools in model theory and number theory, we prove that the theory of $ \Za $ is decidable when $ \alpha $ is a computable transcendental number. To our knowledge, in the latter case (transcendental $ \alpha $) nothing is known about the decidability of the more general structure $ \mathcal{R}_\alpha $, and we believe that our result forms an important step towards solving that problem as well.

We will discuss the definable sets of $ \Za $ later in Subsection \ref{subdefinable}. But, to have a general picture of the model theory involved in $ \Za $, note in particular that $ \Za $ is a model of the theory of $\ZZ$-groups, or Presburger arithmetic without  order, and hence it defines all congruence classes or arithmetic progressions. The latter  are a typical  example of definable sets in $ \Za $ with a somewhat ``structured'' nature.

On the other hand, there are definable sets in $ \Za $ with a ``random'' behaviour, and these sets are typically the subsets of $ \Za $ defined using the powers of the function $ f $. An aspect of this randomness is reflected in the fact that these sets do not contain an infinite 
arithmetic progression (see Subsection \ref{subdefinable} for more details). However, this random behaviour is actually a consequence of definability of a linear order which turns out to be dense by means of Kronecker's approximation lemma (Fact \ref{Krock}).

When applied to the simple case of an irrational number $ \alpha $, Kronecker's lemma says that the set of decimal parts of the sequence $ \{\alpha n:n\in\NN\} $ is dense in the unit interval $ (0,1) $.

However, as it will be expanded later, we need the full strength of Kronecker's lemma. In fact, when $ \alpha $ is transcendental, all different powers of $f$ take an independent part in creating the random behaviour; this is due to independence of all different powers of $ \alpha $ acting underneath the language. This in fact contrasts the case of a quadratic $ \alpha $ where this randomness is actually weaker and dependent only on the first power of $\alpha$, mainly because all terms of the language reduce to a linear form; for example, $ f^2(x) $ is equal to $ f(x)+x-1 $ when $ \alpha $ is the golden ratio (see \cite{af-moh}).
\par 
The following motivating propositions--that won't directly be used in the paper--reflect 
the basic idea on which most of our arguments rely. The idea is
to solve a given  system of equations in $\Za$  by turning, whenever possible, equations into
inequalities in terms of the decimal parts.
The solution to the system will then be obtained by
 an application of $k$-dimensional Kronecker's lemma, Fact \ref{Krock}, asserting the density of  the set $\{([\beta_1 n],\ldots,[\beta_k n])\}_{n\in \mathbb{N}}$
in $(0,1)^k$ for a 
$\QQ$-linear independent set of  real numbers 
$\beta_1,\ldots,\beta_k$.  Recall that the decimal part of a real number $ x $ is denoted by $ \fp{x}, $ namely $ \fp{x}:=x-\ip{x} $. 

\begin{prop}\label{prop-Range}
 Let $ a,b\in \ZZ$:
\begin{itemize}
\item[(1)] $ a $ is in the range of $ f $ if and only if $ \fp{\frac{a}{\alpha}} $ is greater than $ 1-\frac{1}{\alpha} $. Moreover, $ a=f(b) $ implies that $ a=\ip{\frac{b}{\alpha}}+1 $. If $ \alpha $ is positive and less than $ 1 $, then $ f $ is a surjection.

\item[(2)] $f(a+b) = f(a)+f(b)+\ell$ for some $\ell\in\{0,1\}$, where $\ell$ is equal to $0$ if and only if $\fp{\alpha a}+\fp{\alpha b}<1$.
\end{itemize}
\end{prop}

\begin{proof}
	Part (2) is obvious, for part  (1) use the fact that $ a=f(b) $ implies $ \alpha b -1<a<\alpha b. $
\end{proof}

\begin{fact}[Kronecker's approximation lemma, one dimensional]
	\label{Krock-simple}
	The set of decimal parts of the sequence $ \{\alpha n:n\in\NN\} $ is dense in the unit interval $ (0,1) $.
\end{fact}

\begin{prop}\label{prop-Idea}
	The following system of  congruence equations has infinitely many solutions in $\mathbbm{Z}$,
	\begin{align*}
\begin{cases}
		x\stackrel{m}{\equiv}i\\
		f(x)\stackrel{n}{\equiv}j.
	\end{cases}
	\end{align*}
where $x$ is the variable, and $m,n,i,j\in \NN$.
\end{prop}

\begin{proof} 
	It is easy to verify that $ f(x)\stackrel{n}{\equiv}j $ if and only if $ [\frac{\alpha}{n}x]\in(\frac{j}{n}, \frac{j+1}{n}) $. Hence, to solve the system above, it suffices to find $ y\in \ZZ $ such that 
	\[  \fp{\frac{(my+i)\alpha}{n}}\in(\frac{j}{n}, \frac{j+1}{n}). \]
	If $ \frac{j+1}{n}>[\frac{i}{n}\alpha] $, by Fact \ref{Krock-simple} we choose $ y $ such that
	\begin{align*}
		&\fp{\frac{m\alpha}{n}y}<1-[\frac{i}{n}\alpha], \quad\text{and}\\
		&\frac{j}{n}-\fp{\frac{i}{n}\alpha}<\fp{\frac{m\alpha}{n}y}<\frac{j+1}{n}-\fp{\frac{i}{n}\alpha}.
	\end{align*}
	In the case that $ \frac{j+1}{n}<[\frac{i}{n}\alpha] $, again by Fact \ref{Krock-simple} we choose $ y $ such that
	\begin{align*}
		&\fp{\frac{m\alpha}{n}y}>1-\fp{\frac{i}{n}\alpha}, \quad\text{and}\\
		&1+\frac{j}{n}-\fp{\frac{i}{n}\alpha}<\fp{\frac{m\alpha}{n}y}<1+\frac{j+1}{n}-\fp{\frac{i}{n}\alpha}.
	\end{align*}
\end{proof}
We will gradually present a   natural theory $ \mathcal{T}_\alpha $ for 
$ \Za $ by introducing each axiom or axiom scheme
right after
proving the related properties for $ \Za $. 
So, each of the sections below will contain axioms that ensure a partial model-completeness for the final theory $ \Ta $.
In addition, we will justify how incorporating the further assumption of computability for $\alpha$ guarantees the recursive enumerability of our axioms.
We finish by our Main Theorem showing that $ \Ta $ is model-complete and, for a computable $\alpha$,  this suffices for $ \Za $ to be decidable. 
\par 
We find it helpful to give
a summary of our arguments leading to the proof of model-completeness as follows:
\vspace*{4pt}
\begin{itemize}
	\item[Step 1.] We treat certain relations among the decimal parts as first-order $ \mathcal{L} $-formulas (Section \ref{secDecimals}).
	
	\item[Step 2.] We divide systems of equations in $ \mathcal{L} $ into two main categories of non-algebraic (Section \ref{secNonAlg}) and algebraic formulas (Section \ref{secAlg}). 
	
	\item[Step 3.] We use an extended version of Kronecker lemma (Theorem \ref{thm-ExtendedKronecker}) to show that solvability of a system of non-algebraic formulas is equivalent to a quantifier-free $ \mathcal{L} $-formula (Theorem \ref{thm-NAlgModelcomplete}).
	
	\item[Step 4.] For two models $ \mathcal{M}_1\subseteq\mathcal{M}_2 $, we show that the solution, in $ \mathcal{M}_2 $, of an algebraic system involving only a single variable and with parameters in $ \mathcal{M}_1 $, belongs also to $ \mathcal{M}_1 $ (Lemma \ref{lma-Alg-v1}). This implicitly shows that a substructure of a model of $ \Ta $ is closed under taking inverse of different powers of $ f $.
	
	\item[Step 5.] We use a technical trick (Lemma \ref{lma-Alg-technical}) to show that an algebraic system which contains more than two variables reduces to a non-algebraic system of smaller number of variables.

\end{itemize}

Our last section will contain some additional observations and remarks.
\vspace*{7pt}

\begin{convention}\hfil
	\begin{itemize}
		\item[(1)] $\alpha$ is a fixed computable transcendental number. 
		\item[(2)] We will be working in the language $ \mathcal{L}=\{+,-,0,1,f\} $ and unless we state otherwise all formulas are assumed 
		to be in $ \mathcal{L} $. In particular, all axioms and axiom schemes are $ \mathcal{L} $-formulas.
		\item[(3)] 	 When there is no mention of a model or a theory,
		the lemmas and theorems below
		concern the structure 
		$\mathcal{Z}_\alpha$, and hence
		the
		variables and parameters will range over the set of integers $ \ZZ $.
		
		\item[(4)] We occasionally use a finite partial type $ \Gamma(x) $ as a conjunction of $ \mathcal{L} $-formulas as well; that is, we freely use notations like $ \exists x\Gamma(x) $ instead of writing $ \displaystyle{\exists x\bigwedge_{\varphi(x)\in\Gamma(x)}\varphi(x)} $.
		
	\end{itemize}

\end{convention}

\subsection*{A note on overlapping results.}
Short before submitting this paper, we learnt of a similar independent work by 
G\"{u}naydin and \"{O}zsahakyan uploaded in arXiv \cite{gunayden-ozsahakyan}.
The main difference
between the two papers, is that in \cite{gunayden-ozsahakyan}
the authors consider the Beatty sequence as a
predicate in the language, where we
put the function $f=\ip{\alpha x}$,
which is not definable in their structure. For the same reason, we have
to deal with decimals that concern the powers of the function $f$ where
they need only to treat decimals of linear combinations.

\section{Describing  decimals in $\mathcal{L}$}\label{secDecimals} 	
Although our language/theory does not literally contain the decimals, similar observations to Proposition \ref{prop-Range} show that our theory is expressive enough to describe key properties of decimals by capturing their order and dense distribution in the spirit of Fact \ref{Krock}.

\begin{lem}	\label{lma-BasicOperation1}
		Let $ a,b\in\ZZ $ and $ n\in\NN $ with $n\neq 0$.
	\begin{enumerate}

		\item[(1)] 
	There exists $ i\in\{0,\ldots, n-1\} $ such that  
		\[ f(na)=nf(a)+i. \]
	Indeed $f(na)=nf(a)+i$ if and only if  $\frac{i}{n}<\fp{\alpha a}<\frac{i+1}{n} $.
		\item[(2)]
		$[\alpha a]+[\alpha b]<1$ if and only if $ f(a+b)=f(b)+f(b). $

		\item[(3)] 
		$[\alpha a]<[\alpha b]$ if and only if $ f(b-a)=f(b)-f(a) $.				
						
	\end{enumerate}
\end{lem}
\begin{proof}
	Easy to verify. 
\end{proof}	

\begin{axiomsc}[{\footnotesize for various $n\in\NN$ and $0\leq i< n$ with $f(n)\stackrel{n}{\equiv} i$}]\label{ax-cut}
	\[ f(\underbrace{1+\cdots+1}_{n \text{ times}})=\underbrace{f(1)+\cdots+f(1)}_{n \text{ times}}+\underbrace{1+\cdots+1}_{i \text{ times}}. \]
\end{axiomsc}
\begin{rem}\label{rem-prime-computability-cut}\hfill
	\begin{itemize}

	\item[(1)] We can use Axiom scheme \ref{ax-cut} together with part (2) of Axiom \ref{axBasicDLO} below to precisely determine in our theory the value of $f(n)$ for every $n\in\ZZ $. This property is crucial for our theory in order to have $\Za$ as its prime model, and the latter will, in turn, be required for proving the completeness of our theory $\Ta$.
	
	\item[(2)] 	By part (1) of Lemma \ref{lma-BasicOperation1}, we see that the rational cut of $\alpha$ is included in Axiom scheme \ref{ax-cut}. Hence, the computability of $\alpha$ is a necessary condition for this axiom scheme to be recursively enumerable.
	
	\end{itemize}
\end{rem}

\vspace*{1cm}

\begin{lem}\label{lma-BasicOperation2}
	Let $ m,\ell_1,\ell_2\in\ZZ $ and $n\in \NN$.
	\begin{enumerate}
		
	\item[(1)]  There exists a quantifier-free $\mathcal{L}$-formula $ \varphi(x,y) $ such that for all $a,b\in \ZZ $, $\Za\models \varphi(a,b) $ if and only if
	\begin{align}\label{eq-formula-mn}
		\fp{\alpha a} <\frac{m}{n}\fp{\alpha b}.
	\end{align}
		
	\item[(2)]  There exists a quantifier-free $\mathcal{L}$-formula $ \varphi(x,y) $ such that  for all $a,b\in \ZZ$, $\Za\models \varphi(a,b) $ if and only if
		\begin{align}\label{eq-formula-lmn}
			\ell_1<n\fp{\alpha a}+m\fp{\alpha b}<\ell_2.
		\end{align}
	
	\end{enumerate}

\end{lem}
\begin{proof}
	For part (1), first suppose that we have $ 0<m<n $;	other cases can be proved similarly or else they turn into triviality.
	
	Note that there are  $ i , j\in \NN $ with $ 0\leq i<n $ and $ 0\leq j<m $ such that 
	$ \fp{\alpha a}\in(\frac{i}{n},\frac{i+1}{n}) $ and $ \fp{\alpha b}\in(\frac{j}{m},\frac{j+1}{m}) $. The latter conditions are $ \mathcal{L} $-expressible using part (1) of Lemma \ref{lma-BasicOperation1}. Also, observe that non-trivial cases only happen when $ i< m $ and $ j=i $; that is, both $ \fp{\alpha a} $ and $ \frac{m}{n}\fp{\alpha b} $ belong to the same interval $ (\frac{i}{n},\frac{i+1}{n}) $. In this case, using parts (3) and (1) of Lemma \ref{lma-BasicOperation1}, we can show that $ n\fp{\alpha a} < m\fp{\alpha b}$ if and only if $ f(mb-na)=mf(b)-nf(a) $. 
	
	For part (2), note that, depending on whether $ m $ is negative or positive, non-trivial cases occur only when we have that $ -m\leq \ell_1<\ell_2\leq n $ or $ 0\leq \ell_1<\ell_2\leq n+m $, respectively. Then, apply an argument similar to the proof of part (1).
\end{proof}	

The following corollary is a consequence of part (2) of Lemma \ref{lma-BasicOperation2}:
\begin{cor}\label{crl-BasicProperties}
Let $ \ell_1,\ell_2\in \ZZ$, and $ \bar{m} , \bar{n} \in \ZZ^k$. Then, there exists a quantifier-free $ \mathcal{L} $-formula $ \varphi(\bar{x},\bar{y}) $, depending on $k, \ell_1,\ell_2, \bar{m} $ and $ \bar{n} $, such that 
for all $ \bar{a},\bar{b}\in \ZZ^k $,
$ \Za\models  \varphi(\bar{a},\bar{b})$ if and only if 
	\begin{align}\label{eq-MostGeneralExpressible}
		\ell_1<\sum_{i=1}^{k} n_i\fp{\alpha a_i}+\sum_{i=1}^{k}m_i\fp{\alpha b_i}<\ell_2.
	\end{align}
\end{cor}

\begin{notation}\label{not-DecimlFormulas}
	From now on and by injecting suitable variables, we treat various forms of inequalities \eqref{eq-formula-mn}, \eqref{eq-formula-lmn}, \eqref{eq-MostGeneralExpressible} and those appeared in Lemma \ref{lma-BasicOperation1} as first-order $ \mathcal{L} $-formulas. For example, by writing $ \ell_1<n\fp{\alpha x}+m\fp{\alpha y}<\ell_2 $ we mean the quantifier-free $ \mathcal{L} $-formula $\varphi(x,y)$, obtained in part (2) of Lemma \ref{lma-BasicOperation2}.
\end{notation}

Using our observations so far and by applying our flexible notation, we are able to encode into our language some of the main features of the decimals involved in the structure $ \Za $.

\begin{rem}
	Note again that, to a decimal like $ \fp{\alpha a} $ nothing can be designated  in $ \mathcal{L} $. In fact, $ \mathcal{L} $ is capable of describing these decimals merely when they participate in an inequality of the form (\ref{eq-MostGeneralExpressible}).
\end{rem}

In Section \ref{secNonAlg}, 
we deal with axiomatizing the properties of \textit{non-algebraic} formulas, that is, formulas containing decimals of the form $[\alpha f	^i(x)]$ (see Definition \ref{dfn-NonAlgFormula}). A typical example is the formula $ [\alpha x]\in(0,1-\fp{\alpha a}) $ which, according to Axiom scheme \ref{axKronecker} below, is satisfied by infinitely many elements $ x $.

Section \ref{secAlg}, in contrast, axiomatizes the properties of \textit{algebraic} formulas. Here, typical examples include 
formulas like $f(x)=a$ which has finitely many solutions in any model of $ \Th(\Za) $ containing the parameter $ a $ (and indeed a unique solution when $\alpha>1$). 

\section{Non-algebraic formulas}\label{secNonAlg}
	
As mentioned earlier, some of the essential properties of the function $f$ will be described as a consequence of Kronecker's approximation lemma,
and	our aim is to exploit the full extent of
this fact in our axiomatization. A proof of this theorem can be found in \cite[Theorem 442]{Hardy}. 
	
	\begin{fact}[{Kronecker's Approximation  Lemma}]\label{Krock}\hfil
		
Let $ n\in\NN $ and let $\beta_1,\ldots,\beta_n\in \RR$ be 
such that $\beta_1,\ldots,\beta_n,1$ are
linearly independent over $\mathbbm{Q}$.
Then the set
		$\big\{([\beta_1 x],\ldots,[\beta_n x]):{x\in \mathbbm{N}}\big\}$ is dense in $(0,1)^n$.
	\end{fact}
As a direct result of Kronecker's lemma, we can readily verify that \( \Za \) satisfies part (4) of the subsequent axiom. However, it is essential to remember Notation \ref{not-DecimlFormulas}, and in particular, that ``\( \fp{\alpha x} < \fp{\alpha y} \)'' is simply our notation for expressing the formula \( f(y-x) = f(y) - f(x) \). Furthermore, it is straightforward to confirm that \( \Za \) conforms to the remaining three parts of the axiom.
\begin{axiom}\label{axBasicDLO}\hfill
	
		\begin{enumerate}
		\item[(1)] $\forall x\forall y \big(f(x+y)=f(x)+f(y)\ \vee\  f(x+y)=f(x)+f(y)+1\big)$.
		
		\item[(2)] $\forall x\big( f(-x)=-f(x)-1\big)$.
		
		\item[(3)] $\forall y\exists x \big(\bigvee_{i=0}^{f(1)}y+i=f(x)\big)$.
		
		\item[(4)]  The relation $ [\alpha x] <[\alpha y] $ is a dense linear order.
	\end{enumerate}
\end{axiom}
\begin{rem}\label{remBasic}
	Based on our notation, part (1) of Axiom \ref{axBasicDLO} logically implies that 
	for all $ x_1,\ldots, x_n $ we have
	
	\[\bigvee_{j=0}^{n-1} j<\sum_{i=1}^{n}[\alpha x_i]<j+1. \]
\end{rem}
\begin{proof}
With regard to Notation \ref{not-DecimlFormulas}, the formula above is simply another way of expressing the following formula:
\[ \bigvee_{j=0}^{n-1} f(x_1+\cdots+x_n)=f(x_1)+\cdots+f(x_n)+j, \]
which can be demonstrated by $ n $-time applying part (1) of Axiom \ref{axBasicDLO}.
\end{proof}
	
\begin{rem}
	
		An interesting, but not directly relevant, observation
	is that each $n\ZZ$ is dense in $\ZZ$
	with respect to the order relation defined by the decimal parts (Axiom \ref{axBasicDLO}). More generally, given any congruence class
	$ n\ZZ+i $ and any interval 
	$ I=(\fp{\alpha a},\fp{\alpha b})\subset(0,1)$, there are infinitely many integers  $ c\in n\ZZ+i $	such that $ \fp{\alpha c}\in I $. This can be proved 
	using the ideas appeared in the proof of Proposition \ref{prop-Idea}.
\end{rem}
	
	\par 
For given $ a,b\in \ZZ $,  decimals other than $ \fp{\alpha a} $ and $ \fp{\alpha b} $ also intervene in investigating formulas like $ f^n(a+b)=f^n(a)+f^n(b)+\ell $ where the exact value of $ \ell $ is determined by the values of the decimals below:
\[  \fp{\alpha a}, \fp{\alpha b}, \fp{\alpha f(a)}, \fp{\alpha f(b)}, \cdots, \fp{\alpha f^{n-1}(a)}, \fp{\alpha f^{n-1}(b)}.  \]

	To engage with such decimals we aim to expand on the idea of the proof of Proposition \ref{prop-Idea} further as follows: Let $ f^0(x)=x $ and notice that since $\alpha$ is transcendental, each finite sequence $1, \alpha, \ldots, \alpha^{n}$ is $\QQ$-linearly independent. This fact provides us with 
	a more amount of  ``control'' over the decimals of the form $ \fp{\alpha f^i(x)} $ as in the following theorem and the related axiom scheme. 
	 The following theorem concerns the natural numbers and
	the word ``dense'' there has its 
	usual meaning in the reals.
	\begin{thm}[Extended Kronecker's Lemma]\label{thm-ExtendedKronecker}\hfil
		
		For every $ n\in\NN $, the following set of $(n+1)$-tuples is dense in $(0,1)^{n+1}$:
\[ \Big\{\left([\alpha a],[\alpha f(a)],[\alpha f^2(a)],\cdots, [\alpha f^n(a)]\right):a\in \NN\Big\}. \]
	\end{thm}
	\begin{proof}
		For $a\in \NN$,
		note that $ \alpha f(a) =\alpha ^2 a -\alpha\fp{\alpha a}$,
		and an easy induction shows that for every $ k\geq 2 $ we have that 
		\[ \alpha f^k(a) = \alpha ^{k+1}a-\alpha\fp{\alpha^ka}+\alpha (j_1+\cdots+j_{k-1}),\]
		where each $ j_i $ is an integer with $ |j_i|\in\{0,\ldots,\ip{\alpha}\} $ determined by the inequality $ \frac{j_i}{\alpha}<\fp{\alpha^i a}<\frac{j_i+1}{\alpha} $.
	\par
The remainder of the proof is based on the following straightforward observation: Consider \( J_1, \ldots, J_n \) as given subintervals of 
\( (0,1) \). Given any subinterval \( I_0 \) of \( (0,1) \) with sufficiently small length, it is possible to find subintervals \( I_1, \ldots, I_n \) of $(0,1)$ such that for each \( k = 1, \ldots, n \), the cumulative sum \( I_0 + I_1 + \ldots + I_k \) is contained within \( J_k \). Here, \( I + J \) means  the set \( \{\fp{x+y} \mid x \in I, y \in J\} \).
\par 
	Now, suppose that we are to find an \( a\in \NN \) such that \( \fp{\alpha f^i (a)} \) belongs to the specified subinterval \( J_i \) of \( (0,1) \) for each \( i \) in the set \( \{0,\ldots, n\} \). To achieve this, we only need to select a sufficiently small subinterval \( I_0 \) of \( J_0 \), determine intervals \( I_1,\ldots,I_n \) as described earlier, and find an element \( a\in\NN \) such that \( [\alpha^i a] \) belongs to \( I_i \) for each \( i \). The existence of such \( a \) is guaranteed by Kronecker's approximation theorem.	
	\end{proof}
\begin{axiomsc}[{\footnotesize for various $m,n\in\NN$}]\label{axKronecker}\hfil
	
		For every $ \bar{y} $ and $ \bar{z} $ (with $ |\bar{y}|=|\bar{z}|=n $), if 
		$\displaystyle{\bigwedge_{i=1}^n[\alpha z_i]<[\alpha y_i]}$, then
		there exist at least $m$ different elements $x$ such that
		\[
		\bigwedge_{i=1}^n \quad [\alpha y_i]< [\alpha f^i(x)]< [\alpha z_i].
		\]
\end{axiomsc}

\begin{dfn}\label{dfn-NonAlgFormula}\hfil
	
	\begin{itemize}
		\item[(1)] 
		Regarding Notation \ref{not-DecimlFormulas}, for every $ q\in\NN, \ell\in\ZZ $ and tuples of integers $ \bar{m} $ and $ \bar{n} $, any quantifier-free $\mathcal{L}$-formula $ \theta(x;\bar{y}) $ of the form below is called \textit{non-algebraic}:
		\vspace*{3pt}
		\begin{equation}\label{eq-NonAlgOneVar}
			\begin{aligned}
				n_0\fp{\alpha x}+n_1\fp{\alpha f(x)}+\ldots+ &n_k\fp{\alpha f^k(x)} \text{\large$\vartriangleleft$}\\	
				&m_0\fp{\alpha y_0}+\ldots+ m_k\fp{\alpha y_k}+\fp{\alpha q}+\ell,
			\end{aligned}	
		\end{equation}
		
		\vspace*{5pt}
		where $ |\bar{y}|= |\bar{m}|=|\bar{n}| =k+1 $ and $\vartriangleleft$ is either of the symbols ``$<$'' or ``$>$''.
		
		 More generally, define a \textit{non-algebraic} formula $\theta(\bar{x};\bar{y})$ with $ |\bar{x}|=k $ and $ |\bar{y}|=k'+1 $ 
		as a quantifier-free $\mathcal{L}$-formula of the following form
		\vspace*{3pt}
		\begin{align}\label{eq-NonAlg}
		\hspace*{35pt} \sum _{i=0}^{\ell_1}n_{1i}\fp{\alpha f^i(x_1)}+\cdots+\sum _{i=0}^{\ell_k}n_{ki}\fp{\alpha f^i(x_k)}\text{\large$\vartriangleleft$}\sum _{i=0}^{k'}m_i\fp{\alpha y_i}+\fp{\alpha q}+\ell	
		\end{align}
		\vspace*{3pt}
		for some $ \ell_1,\ldots,\ell_k,q\in\NN $ and $\ell,m_i,n_{ji}\in \ZZ$ and $\vartriangleleft$ as above.	
		
		\item[(2)] Let $ \mathcal{M} $ be an $ \mathcal{L} $-structure, $ A\subseteq M $ and $\bar{a}\in M$. The \textit{non-algebraic type} of $\bar{a}$ over $A$ is the partial type consisting of all
		non-algebraic formulas $\theta(\bar{x};\bar{b})$, with $ |\bar{x}|=|\bar{a}| $ and $ \bar{b}\in A $, which are satisfied by $ \bar{a} $ in $\mathcal{M}$.
	\end{itemize}

\end{dfn}
%
%
%

Hence, a given non-algebraic type $ \pi(x;b) $ over a single parameter $ b $ determines, among other things, the value of the numbers  
$\ell_i\in \{0,1\}$
such that 
\begin{align}\label{eq-NonAlgExample}
	\begin{split}
	f(x+b)&=f(x)+f(b)+\ell_1 \\ 
	f^2(x+b)&=f(f(x)+f(b)+\ell_1)=f^2(x)+f(f(b)+\ell_1)+\ell_2 \\
	f^3(x+b)&=f(f^2(x)+f(f(b)+\ell_1)+\ell_2)\\
	&=f^3(x)+f(f(f(b)+\ell_1)+\ell_2)+\ell_3 \\
	& \vdots
	\end{split}
\end{align}	
which are included in $\pi(x;b)$ as the following set of inequalities assuming $ \ell_i=0 $ (respectively $ \ell_i=1 $).	
\begin{align*}
	&[\alpha x]+[\alpha b]<1 &\quad(>1)\\
	&[\alpha f(x)]+[\alpha (f(b)+\ell_1)]<1 &\quad(>1) \\
	&[\alpha f^2(x)]+[\alpha (f(f(b)+\ell_1)+\ell_2)]<1 &\quad(>1)\\
	&\vdots &\vdots .
\end{align*}

Lemma \ref{lma-NonAlgQE} below provides a quantifier-free condition for a non-algebraic type to have a solution. To get the idea of its proof via an example,
 consider an inequality of the form \eqref{eq-NonAlg}:
 	\[ 2\fp{\alpha x}-3\fp{\alpha f(x)}< 4\fp{7\alpha } - 6 \fp{9\alpha}+\fp{8\alpha}+5, \]
 	where $m_0=4, m_1=-6, y_0=7, y_1=9, q=8$, and $\ell = 2$. 
%
	According to Theorem \ref{thm-ExtendedKronecker}, it suffices to find a pair $ (z_0,z_1)\in(0,1)^2 $ satisfying 
	\[ 2z_0-3z_1 < 4\fp{7\alpha } - 6 \fp{9\alpha}+\fp{8\alpha}+5. \]
	
	But, the existence of a solution $ (z_0,z_1)\in(0,1)^2 $ for this inequality is simply equivalent, in $\RR$, to the requirement that
		\begin{align}\label{eq-computability}
		-3<4\fp{7\alpha } - 6 \fp{9\alpha }+\fp{8\alpha}+5 <2.	
		\end{align}
		
This is indeed the Fourier-Motzkin elimination algorithm for solving system of linear inequalities in real numbers. 
\begin{lem}\label{lma-NonAlgQE}
	Suppose that $ \Gamma(x;\bar{y}) $ is a finite set of non-algebraic formulas  each of the form (\ref{eq-NonAlgOneVar}). Then, there exists a quantifier-free formula $ \chi(\bar{y}) $, depending on the numbers $ k, \ell, \bar{n} $ and $ \bar{m} $ appearing in the formulas $ \theta\in\Gamma $, such that 
	\begin{align*}
		\Za\models \forall\bar{y}\left(\exists x\Gamma(x;\bar{y})\leftrightarrow\chi(\bar{y})\right).	
	\end{align*}
\end{lem}
\begin{proof}
First consider the case that $ \Gamma(x;\bar{y}) $ only consists of a single formula $ \theta(x;\bar{y}) $. Let $ w_\theta $ denote the sum of decimals appearing in the right-hand side of inequality (\ref{eq-NonAlgOneVar}). Based on Theorem \ref{thm-ExtendedKronecker}, we can replace each $\fp{\alpha f^i(x)} $ by a new real-valued variable $z_i$ ranging over $(0,1)$. 

Now, working in $\RR$, fix $ w_\theta $ and note that the existence of a real-valued solution $ (z_0,\cdots,z_k)\in(0,1)^{k+1} $ for the linear inequality 
$\sum_{i=0}^{k}n_iz_i<w_\theta+\ell_\theta$
is equivalent to the quantifier-free formula $\chi_{_\theta}$ asserting $s_\theta<w_\theta+\ell_\theta<t_\theta$ where $s_\theta$ and $t_\theta$ are respectively the minimum and the maximum of the summation $\sum_{i=0}^{k}n_iz_i$ with $z_i\in (0,1)$ for each $i$. Note that the value of $s_\theta$ and $t_\theta$ are determined by the coefficients $n_i$.

In the case that $ \Gamma(x;\bar{y}) $ contains more than one formula, we proceed as above by introducing the corresponding real variables $ z_i $ for the decimals $ \fp{\alpha f^i(x)} $. For each of the formulas $ \theta(x;\bar{y})\in\Gamma(x;\bar{y}) $, we need to consider the corresponding $ w_\theta $ as above, and the required quantifier-free formula is the conjunction of the formulas $\chi_{_\theta}$. Note that, for different formulas $\theta,\theta'\in\Gamma$ the corresponding $w_\theta$ and $w_\theta '$ might be interrelated since each $w_\theta $ is a linear combination of multiple variables of the form
$ m_0\fp{\alpha y_0}+\ldots+ m_k\fp{\alpha y_k} $.
\end{proof}

\begin{cor}\label{crl-NonAlgQE}
		Suppose that $ \Gamma(\bar{x};\bar{y}) $ is a finite set of non-algebraic formulas  each of the form (\ref{eq-NonAlg}). Then, there exists a quantifier-free formula $ \chi(\bar{y}) $, depending on the numbers $\ell_1,\ldots,\ell_k, \ell, m_i, n_{ji}$ appearing in the formulas $ \theta\in\Gamma $, such that the following holds in $ \Za $: 
	\begin{align}\label{eq-NonAlgQE}
		\forall\bar{y}\left(\exists\bar{x}\Gamma(\bar{x};\bar{y})\leftrightarrow\chi(\bar{y})\right).	
	\end{align}
\end{cor}
\begin{proof}
	Proceeding as in the proof of Lemma \ref{lma-NonAlgQE}, and replacing each $\fp{\alpha f^i(x_j)} $ appearing in $\Gamma(\bar{x};\bar{y})$ by a new real-valued variable $z_{i,j}$, we can use Fourier-Motzkin algorithm to compute the suitable $\chi(\bar{y})$.
\end{proof}
\begin{axiomsc}\label{axNonAlg}
	All instances of formula (\ref{eq-NonAlgQE}) above when $ \Gamma(\bar{x};\bar{y}) $ and $ \chi(\bar{y}) $ range over all formulas having the properties described in Corollary \ref{crl-NonAlgQE}.
\end{axiomsc}

\begin{rem}\label{rem-computablity-non-alg}
A simple instance of the axiom scheme above is a formula like
\begin{align*}
	\exists x \Big(\fp{\alpha x}<1-\fp{\alpha} \ \wedge \ \fp{\alpha x}<1-\fp{2\alpha}\Big) \quad \leftrightarrow \quad 
	0<\fp{2\alpha}-\fp{\alpha}<1.
\end{align*}
    It is worth noting that the computability of $\alpha$ does not play a role in recursively listing such sentences.
\end{rem}
\begin{notation}\hfil
	
	\begin{itemize}
		\item[(1)] 
		Let $T_0$ denote the theory of $\mathbb{Z}$-groups
		(that is, Presburger arithmetic without order) in the language $ \{0,1,+,-\} $.		
		\item[(2)] 
		Let $ \Tnalg $, reads ``$T$-non-algebraic'', be $T_0$ together with Axiom \ref{axBasicDLO}, Axiom scheme \ref{axKronecker}, and Axiom scheme \ref{axNonAlg}.
	\end{itemize}
	
\end{notation}

\begin{thm}
	\label{thm-NAlgModelcomplete}
	Suppose that $\mathcal{M}_1\subseteq \mathcal{M}_2$ are models of 
	$\Tnalg$ and let $ A\subseteq M_1 $ and $ \bar{a}\in M_2 $. Then, every finite fragment of the non-algebraic type of 
	$ \bar{a} $ over $ A $ is realized in $ \mathcal{M}_1 $.
\end{thm}
\begin{proof}
	According to Axiom scheme \ref{axNonAlg}, the existence of a solution for each finite fragment of the non-algebraic type of $ a $ over $ A $ is equivalent to the satisfaction of a quantifier-free formula $ \chi(\bar{y}) $ by a finite tuple $ \bar{b}\in A $. But, $ \chi(\bar{b}) $ holds in $ \mathcal{M}_2 $ if and only if it holds in $ \mathcal{M}_1 $.

\end{proof}

For an interesting connection of 
$\Tnalg$ to o-minimality, see Subsection \ref{subsecOMin}.

\section{Adding algebraic formulas}\label{secAlg}
In this section, we focus on the most general form of a finite set of $ \mathcal{L} $-formulas by adding formulas of the form $ h(x)=y $ where $ h(x) $ is a term of the form 
\begin{align}\label{eq-Alg-term}
	 \sum_{i=0}^{k}m_if^i(x),
\end{align}
for some $ m_i\in\ZZ $; we will refer to these integers as \textit{coefficients of} $ h(x) $.
Also we call a term of the form \eqref{eq-Alg-term} an \textit{$f$-polynomial.}

\begin{dfn}\label{dfn-Alg}
	An $ \mathcal{L} $-formula $ \varphi(x_1,\ldots,x_n;y) $ is called \textit{algebraic} if it is of the  form 
	\[ h_1(x_1)+\ldots+h_n(x_n)=y, \]
where $ h_1(x_1), \ldots, h_n(x_n) $ are $f$-polynomials.
\end{dfn}

\begin{lem}\label{lma-Alg-diff-negative}
	Suppose that $ h(x) $ is an $f$-polynomial. Then, there exists a minimal $ K_h\in\NN $, depending on coefficients appearing in (\ref{eq-Alg-term}), such that:
	\begin{align}
		&\Tnalg\models \forall x_1\forall x_2 \Big(\bigvee_{j=-K_h}^{K_h}h(x_1+ x_2)=h(x_1)+ h(x_2)+j\Big),\label{eq-Alg-diff}\text{ and }\\
		&\Tnalg\models \forall x\Big(\bigvee_{j=0}^{K_h}h(-x)=-h(x)-j\Big).\label{eq-Alg-negative}
	\end{align}
\end{lem}
\begin{proof}
	Based on part (1) of Axiom \ref{axBasicDLO}, we can use a simple induction on $i\in\NN$ to obtain a minimal positive integer $K_i$ such that for every $x_1$ and $x_2$ we have that 
		\[ \bigvee_{j=0}^{K_i} f^i(x_1+x_2)=f^i(x_1)+f^i(x_2)+j. \]
	Now, given an $f$-polynomial $h(x)$ of the form (\ref{eq-Alg-term}), the required minimal $K_h$ is a natural number--depending on the coefficient appearing in $h(x)$--less than or equal to $ \sum_{i=1}^k |m_i|K_i $. This proves \eqref{eq-Alg-diff}.
	
	A similar argument as above and part (2) of Axiom \ref{axBasicDLO} will work for \eqref{eq-Alg-negative}.
	
\end{proof}

\begin{lem}\label{lma-Alg-bounds}
	Let $h(x)$ be an $f$-polynomial. 
	\begin{itemize}
		\item[(1)] For every $n\in\NN$, there exists  $K_n\in\NN$ such that for every $x\in\ZZ$ we have 
		\[f^n(x)=\ip{\alpha ^n x}-j, \]
		for some $j\in\{0,\ldots, K_n\}$.
		
		\item[(2)] There exist $k_h\in\NN$ and $\beta_h\in\RR$ such that for every $x\in\ZZ$ we have 
		\[ h(x)=\ip{\beta_hx}+j, \]
		for some  $j\in \ZZ$ satisfying $-k_h\leq j \leq k_h$. Moreover, $\beta_h$ is equal to $\sum_{i=0}^km_i\alpha^i$ where the $m_i$'s and $k$ are as in 
		\eqref{eq-Alg-term}.
		
		\item[(3)] There exists  $N_h\in\NN$ such that:
		\begin{align}\label{eq-Alg-range}
			&\Za\models \forall y\exists x\Big(\bigvee_{j=0}^{N_h}h(x)=y+j\Big).
		\end{align}
		
		\item[(4)] For every $\ell\in\NN$, there exists an $\ell'\in\NN$ such that 
		\[
		\Za\models \forall x\forall y\ \Big( 
		\big( \bigvee_{j=-\ell}^{\ell}h(x)-h(y)=j\big) \to \big( \bigvee_{j=-\ell'}^{\ell'}(x-y=j\big)
		\Big).
		\]
	\end{itemize}
\end{lem}
\begin{proof}
	Part (1) can easily be proved using a simple induction on $n$, and the fact that $f^{n+1}(x)=\ip{\alpha f^n(x)}$. Part (2) is also obvious using part (1) and the properties of the floor function.
	
	For part (3), note that the function $\ip{\beta_h x}$ is monotone and unbounded. So, given   $y\in \ZZ$, there exists  $x\in \ZZ$ such that 
	$y$ belongs to the interval with endpoints $\ip{\beta_h x}$ and $\ip{\beta_h (x+1)}$. The length of this interval is at most $\ip{\beta_h}+1$. By part (2), there are integers $j_1,j_2\in\{-k_h,\ldots,k_h\}$ such that $\ip{\beta_h x} = h(x)+j_1$ and $\ip{\beta_h (x+1)} = h(x+1)+j_2$. Hence, we can take $N_h$ to be $\ip{\beta_h}+2k_h+1$.
	
	For part (4), we can use part (2) to conclude that $|\beta_h (x-y)|\leq \ell+2k_h+1$ whenever $|h(x)-h(y)|\leq \ell$. Hence, we can take $\ell'$ to be equal to $\frac{\ell+2k_h+1}{|\beta_h|}$.
\end{proof}

\begin{rem}
	Parts (1) and (2) in the lemma above are far from being expressible in $\mathcal{L}$. However, our theory only needs their particular consequences appearing in parts (3) and (4). It is a good question to ask whether or not these properties are provable merely based on the axioms we have introduced so far. However, answering this question does not seem to be quite straightforward, hence, we choose the easier way of adding parts (3) and (4) to our axioms:
\end{rem}

\begin{axiomsc}\label{ax-RangeOfTerms}
	All instances of formula \eqref{eq-Alg-range} for $h(x)$ and $N_h$ as appeared in part (3) of Lemma \ref{lma-Alg-bounds}.
\end{axiomsc}
\begin{axiomsc}[{\footnotesize For each $h(x)$ and $\ell,\ell'\in\NN$ as in part (4) of Lemma \ref{lma-Alg-bounds}}]\label{ax-Bounds}\hfil
	\[ \forall x_1\forall x_2\Big(\bigvee_{j=-\ell}^{\ell}h(x_2)-h(x_1)=j\ \rightarrow\  \bigvee_{j=-\ell'}^{\ell'}x_2-x_1=j\Big). \]
	
\end{axiomsc}
\begin{rem}\label{rem-computability-h}
	As can be seen in the proof of Lemma \ref{lma-Alg-bounds}, the computability of $\alpha$ is required in calculating the numbers $N_h$ and $\ell'$ above, and hence also, in recursively listing all instances of the Axiom schemes \ref{ax-RangeOfTerms} and \ref{ax-Bounds}.
\end{rem}
\begin{notation}
	Let $ \Ta $ denote the $\mathcal{L}$-theory consisting of $ \Tnalg $ together with Axiom schemes \ref{ax-RangeOfTerms} and \ref{ax-Bounds}.
\end{notation}

Next lemma shows that in $\Tnalg$ the composition of $\mathcal{L}$-terms is commutative up to a uniformly bounded error.

\begin{lem}\label{lma-Alg-composition}
 Given two  $ f $-polynomials $h_1(x)$ and $h_2(x)$, there exist a term $h(x)$ and $K_{h_1,h_2}\in \NN$ such that:
\[ \Tnalg \models \forall x\ \Big( (h_1(h_2(x))=h(x)+j_1) \wedge  (h_2(h_1(x))=h(x)+j_2)\Big), \]
where $j_1,j_2\in \ZZ$ are with an absolute value less than or equal to $K_{h_1,h_2}$. 
\end{lem}
\begin{proof}
 Let $ h_1(x)= \sum_{i=0}^{k}m_if^i(x) $ and $h_2(x)= \sum_{j=0}^{k'}n_jf^j(x)$. Using \eqref{eq-Alg-diff}, it is easy to see that we have 
 \[ h_1(h_2(x))=\sum_{\ell=0}^{kk'}q_\ell f^\ell(x)+j_1\ \text{ and } \  h_2(h_1(x))=\sum_{\ell=0}^{kk'}p_\ell f^\ell (x)+j_2,\]
 for some $j_1,j_2\in \ZZ$. Also, it is easy to see that for each $\ell\in\{0,\ldots,kk'\}$ we have that 
  \[ q_\ell=p_\ell=\sum_{i+j=\ell}m_in_j. \]
 We can use the numbers $K_{h_1}$ and $K_{h_2}$ obtained  for $h_1$ and $h_2$ in Lemma \ref{lma-Alg-diff-negative} to calculate a suitable upper bound for $j_1$ and $j_2$.
\end{proof}

\begin{lem}\label{lma-non-alg}
	Suppose that $ h(x) $ is an $f$-polynomial, and $ \theta(x;\bar{y}) $ is a non-algebraic formula of the form (\ref{eq-NonAlgOneVar}). Then, there are finitely many non-algebraic formulas $ \theta_1(x;\bar{y}),\ldots,\theta_n(x;\bar{y}) $, determined by the coefficients appearing in $ h(x) $ and $ \theta(x;\bar{y}) $, such that: 
	\begin{align}\label{eq-Alg-Disjunction}
\Ta\models 		\forall x\forall \bar{y}\left(\theta\big(h(x);\bar{y}\big)\leftrightarrow\bigvee_{i=0}^n\theta_i(x;\bar{y})\right).
	\end{align}
\end{lem}
\begin{proof}
	Using Notation \ref{not-DecimlFormulas} and the discussion after Definition \ref{dfn-NonAlgFormula}, the proof simply follows from the fact that the following formula is derivable from Axiom \ref{axBasicDLO}:
    \[ \bigvee_{\ell=0}^{K} \left(\fp{\alpha\sum_{i=0}^{k}m_if^i(x)} = \sum_{i=0}^{k}m_i\fp{\alpha f^i(x)}+\ell\right), \]
    where $K=\sum_{i=0}^{k}|m_i|$.    	
\end{proof}
\begin{cor}\label{crl-non-alg}
		Suppose that $ \theta(x_1,\ldots,x_m;\bar{y}) $ is a non-algebraic formula of the form (\ref{eq-NonAlg}), and $ h_1(\bar{x}),\cdots, h_m(\bar{x}) $ are terms in the language, where $|\bar{x}|=m$. Then, there are finitely many non-algebraic formulas $ \theta_1(\bar{x};\bar{y}),\ldots,\theta_n(\bar{x};\bar{y}) $, determined by the terms $ h_i(\bar{x}) $ and $ \theta(\bar{x};\bar{y}) $, such that: 
	\begin{align}\label{eq-Alg-Disjunction}
		\Ta\models 		\forall \bar{x}\forall\bar{y}\left(\theta\Big(h_1(\bar{x}),\ldots,h_n(\bar{x});\bar{y}\Big)\leftrightarrow\bigvee_{i=0}^n\theta_i(\bar{x};\bar{y})\right).
	\end{align}
\end{cor}
\begin{proof}
	Using Lemma \ref{lma-Alg-diff-negative}, for each $h_i(\bar{x})$, there are finitely many single-variable terms $h'_{i,j}(x_j)$ and $K_i\in\NN$ such that we have 
	\[ \bigvee_{\ell=-K_i}^{K_i} h_i(x_1,\ldots,x_m) = h'_{i,1}(x_1)+\ldots+h'_{i,m}(x_m)+\ell. \]
	 Now, we can treat each decimal of the form $\fp{\alpha h(x)}$ as in the proof of Lemma \ref{lma-non-alg}.
\end{proof}

Lemmas \ref{lma-Alg-technical}, \ref{lma-Alg-v1}, and \ref{lma-Alg} form our main technical steps towards solving systems involving more than one variable and consisting of at least one algebraic formula.
Lemma \ref{lma-Alg-technical} shows how a finite set of non-algebraic formulas constrained with a single algebraic formula can turn into a non-algebraic system of a smaller number of variables (which was already handled in Section \ref{secNonAlg}). This is actually done at the expense of adding an extra variable, $ w $ in the lemma. This new variable is forced to belong to the preimage of a new term $ h(x) $, and to be obtained using the variables already available in the system.
\par 

The proof we have provided for Lemma \ref{lma-Alg-technical}  
 is only a careful elaboration of a simple idea.
One can initially examine the case of $n=2$ in a model $\mathcal{M}$ of the theory, where the  idea would be as follows: Let $\bar{b},c\in M$, and assume that a finite set of non-algebraic formulas $ \Gamma(x_1,x_2;\bar{b}) $ is given
alongside with an algebraic equation $h_1(x_1)+h_2(x_2)=c$. We find a suitable term $h(x)$ and elements $a$ and $c'$ so as to turn the latter equation into a new one:  $h(x_1+x_2)=c'$, and then into the simpler equation: $x_1+x_2=a$. Finally we replace $x_2$ with
$a-x_1$ in $\Gamma$, which leads to having a totally non-algebraic system with a single variable. As can be seen in the proof, the new elements $a$ and $c'$ as well as the bounds of their whereabouts is uniformly obtained within the theory, and hence it does not depend on the given model $\mathcal{M}$. 

Before we proceed with Lemma \ref{lma-Alg-technical}, we need to introduce an auxiliary notation which will be followed by an easy lemma:


\begin{notation}\label{not-auxil}
	Let $ h_1(x),\ldots, h_n(x) $ be $f$-polynomials. For each $ i\in\{1,\ldots,n\} $ the term $ h^{\ast}_i(x) $ denotes 
	the successive composition of all the terms in the set $ \{h_j(x):j\neq i\} $. For example, $ h^{\ast}_2(x) $ denotes 
	\[ h_1(h_{3}(h_{4}(\cdots(h_n(x))\cdots))).\]
\end{notation}

\begin{lem}\label{lma-Alg-composition-star}
	Suppose that $ h_1(x),\ldots, h_n(x) $ are $f$-polynomials. Then, there exist an $f$-polynomial $ h(x) $ and  $ K=K_{(h_1,\ldots,h_n)}\in \NN $, depending on coefficients appearing in $ h_1(x),\ldots, h_n(x) $, such that 
	\begin{align*}
		\Tnalg\models	\forall x\ 	\bigvee_{\bar{j}\in \mathbf{K}}\hspace*{4pt} \bigwedge_{i=1}^{n}\Big(h_i(h^{\ast}_i(x))=h(x)+j_i\Big),
	\end{align*}
	where $ \bar{j}=(j_1,\ldots,j_n) $ and $  \mathbf{K}=\{0,\ldots,K\}^n $.
\end{lem}
\begin{proof}
	Similar to the proof of Lemma \ref{lma-Alg-composition}.
\end{proof}
In the following lemma, with $\bigwedge \Gamma$, for a set of formulas $\Gamma$, we mean the conjunction of all formulas  in $\Gamma$.
\begin{lem}[Technical Trick]\label{lma-Alg-technical}
	Let $ n\in\NN $, and suppose that $ \Gamma(\bar{x};\bar{y}) $ with \mbox{$ |\bar{x}|=n $} is a finite set of non-algebraic formulas, and $ h_1(x_1),\ldots,h_n(x_n) $ are given $f$-polynomials. Then, there exist an $f$-polynomial  $ h(x) $, a number $N_h\in\NN$, and finitely many sets of non-algebraic formulas $ \Gamma_1(x_1,\ldots,x_{n-1};z,\bar{y}),\ldots, \Gamma_m(x_1,\ldots,x_{n-1};z,\bar{y}) $ such that  
	\begin{align*}
	\Ta\models 	\forall\bar{y}\forall z\exists w\bigvee_{J=-N_h}^{N_h}\bigg(h(w)=z+J\hspace*{3pt}\wedge\hspace*{3pt}
		\Big(\exists \bar{x}\psi(\bar{x};\bar{y},z)\leftrightarrow\exists \bar{x}'\bigvee_{i=0}^m\left(\bigwedge\Gamma_i(\bar{x}';\bar{y},w)\right)\Big)\bigg),
	\end{align*}
	where the symbol $\bar{x}'$ abbreviates the sequence of variables $x_1,\ldots,x_{n-1}$, and $ \psi(\bar{x};\bar{y},z) $ is the following conjunction of formulas:
	
		\begin{align}\label{eq-Alg-tech}
		\bigwedge\Gamma(x_1,\ldots,x_n;\bar{y})\ \wedge \ \sum_{i=1}^n h_i(x_i) = z.
	\end{align}
\end{lem}
\begin{proof}
	For each $i\in\{1,\ldots, n\}$, let $h^{\ast}_i$ be as in Notation \ref{not-auxil}, and $h(x)$ the $f$-polynomial obtained for $h_1,\ldots, h_n$ by Lemma \ref{lma-Alg-composition-star}. Also, let $N_h\in\NN$ denote the number that exists for $h(x)$ by Axiom Scheme \ref{ax-RangeOfTerms}. This axiom scheme also implies the existence of  $j\in\ZZ$ with $|j|\leq N_h$ such that $z+j$ belongs to the range of $h$. In other words, there is $w$ such that $z=h(w)-j$.
	
	Using a similar argument for $f$-polynomials $h^*_i$, for each $i\in\{1,\ldots, n\}$ there are $j^*_i\in\ZZ$ and $x'_i$ with $|j^*_i|\leq N_{h^*_i}$ such that 
	$x_i=h^*_i(x'_i)-j^*_i$. This turns the algebraic part of \eqref{eq-Alg-tech} into:
	\[ \sum_{i=1}^{n}h_i\Big(h^*_i(x'_i)-j^*_i\Big) = h(w)-j. \] 
	
	Using Lemmas \ref{lma-Alg-diff-negative} and \ref{lma-Alg-composition-star}, there are $j_1,\ldots,j_n\in\ZZ$ such that the left-hand side of the equation above turns into:
	\[ \sum_{i=1}^{n}\left(h(x_i)-h_i(j^*_i)+j_i\right). \]
	Definition \ref{dfn-NonAlgFormula} and the proof of Lemma \ref{lma-non-alg} show that each $j_i$ is determined by finitely many non-algebraic formulas, say $\Delta^*_i(x_i)$. Additionally, for each $i\in\{1,\ldots, n\}$ we have $|j_i|\leq K'$ for some $K'\in\NN$ obtained based on the numbers $K_{h_i}$ and $K$ appearing in Lemmas \ref{lma-Alg-diff-negative} and \ref{lma-Alg-composition-star} respectively.
	
	All these turn the initial equation into: 
	\[ h(x'_i+\cdots+x'_n) - h(w) = J,\]
	for some $J\in\ZZ$. The proof of Corollary \ref{crl-non-alg} shows that the exact value of $J$ is determined by finitely many non-algebraic formulas; we denote it by $\Delta(x'_1,\ldots,x'_n)$.
	
	By Axiom Scheme \ref{ax-Bounds}, for some $J'\in\ZZ$ we have that 
	 $x'_1+\cdots+x'_n-w=J'$,
	or 
	\[ x'_n = w+J'-\sum_{i=1}^{n-1}x'_i. \]

Thus replacing $x'_n$ and using Corollary \ref{crl-non-alg}, the conjunction of formulas in $\Delta(x'_1,\ldots,x'_n)$ will be equivalent to a disjunction of finitely many non-algebraic formulas $\Delta'(x_1,\ldots,x'_{n-1},w)$. Similarly, the conjunction of formulas in $\Delta^*_n(x_n)$ introduced above is equivalent to the disjunction of a finite set of non-algebraic formulas which we denote by  $\Delta^*(x_1,\ldots,x'_{n-1},w)$.

Since for each $i\in\{1,\ldots, n\}$ we have $x_i=h^*_i(x'_i)-j^*_i$, we are able to replace each $x_i$ in $\Gamma(x_1,\ldots,x_n;\bar{y})$ accordingly. Hence, by Corollary \ref{crl-non-alg}, each non-algebraic formula in $\Gamma(x_1,\ldots,x_n;\bar{y})$ will be equivalent to a finite disjunction of non-algebraic formulas $\Gamma'(x'_1,\ldots,x'_n;\bar{y})$. By replacing $x'_n$ and using Corollary \ref{crl-non-alg} once again, each formula in $\Gamma'(x'_1,\ldots,x'_n;\bar{y})$ will be equivalent to the disjunction of a set of non-algebraic formulas which we denote by $\Gamma''(x_1,\ldots,x_{n-1},w;\bar{y})$. 

Finally, the following union is the  finite set of non-algebraic formulas required by the lemma:

\[ \bigcup_{i=1}^{n-1}\Delta^*_i(x_i) \cup \Delta'(x_1,\ldots,x'_{n-1},w)\cup \Delta^*(x_1,\ldots,x'_{n-1},w) \cup \Gamma''(x_1,\ldots,x_{n-1},w;\bar{y}).  \]

\end{proof}

The following lemma considers situations where a model contains the solutions of an algebraic formula with a single variable and with parameters from a smaller model. We see that, the smaller model always contains solutions for another formula which is approximately the same as the original one. This surprisingly leads to finding the exact solutions of the original formula already inside the smaller model.  

\begin{lem}\label{lma-Alg-v1}
	Let $\mathcal{M}_1\subseteq \mathcal{M}_2$ be models of 
	$\Ta$. Also, suppose that $ b\in M_1 $ and $a\in M_2$ are such that $h(a)=b$ for some  $f$-polynomial $h(x)$. Then, the element $a$ belongs to $ M_1 $. 
\end{lem}
\begin{proof} 
	By Axiom scheme \ref{ax-RangeOfTerms}, there exist an integer $ j\leq N_h $ and an element $ a'\in M_1 $ such that we have $h(a')=c+j$ in $ \mathcal{M}_1 $. Let $\ell = N_h$ and use Axiom scheme \ref{ax-Bounds} to conclude that $ |a-a'|\leq\ell' $ in $ \mathcal{M}_2 $. Now, the axioms of $ T_0 $ ensure that $ a $ is the $ \ell' $-th successor/predecessor of $ a' $, and this means that $ a $ is already a member of $ M_1 $.
\end{proof}

We combine all the previous results to prove Lemma \ref{lma-Alg}. As we see in the proof of the Main Theorem, the following lemma is in fact the final main step in showing that $\Ta$ is model-complete.


\begin{lem}\label{lma-Alg}
	Suppose that $\mathcal{M}_1\subseteq \mathcal{M}_2$ are models of $\Ta$.
	Also, suppose that $ \bar{b},c\in M_1 $, that $ \Gamma(\bar{x};\bar{b}) $ with $ |\bar{x}|=n $ is a finite set of non-algebraic formulas, and that $ h_1(x_1),\ldots,h_n(x_n) $ are given $f$-polynomials. 
	Then	 the following set of formulas is satisfiable in $ \mathcal{M}_2 $ if and only if it is satisfiable in $ \mathcal{M}_1 $:
	\begin{align}\label{eq-Alg-main-lma}
		\Gamma(x_1,\ldots,x_n;\bar{b})\cup
		\Big\{h_1(x_1)+\ldots+ h_n(x_n)=c\Big\}.
	\end{align}
\end{lem}
\begin{proof}
		By applying Lemma \ref{lma-Alg-technical} for $ \mathcal{M}_2 $, there are $ a\in M_2 $, an $f$-polynomial $ h(x) $, an element $ J\in \ZZ $ with $ h(a)=c+J $, and finitely many sets of non-algebraic formulas 
		$ \Gamma_1(x_1,\ldots,x_{n-1};a,\bar{b}),\ldots, \Gamma_m(x_1,\ldots,x_{n-1};a,\bar{b}) $
		such that satisfiability of the set of formulas (\ref{eq-Alg-main-lma}) in $ \mathcal{M}_2 $ is equivalent to the truth of a disjunctive formula of the form below:
		\[ \bigvee_{i=1}^{m}\left(\bigwedge\Gamma_i(\bar{x}';\bar{y},w)\right). \]
		 Assuming (\ref{eq-Alg-main-lma}) is satisfiable in $ \mathcal{M}_2 $, at least one of the mentioned sets, say $ \Gamma_i(x_1,\ldots,x_{n-1};a,\bar{b}) $, is satisfiable in $ \mathcal{M}_2 $.
	
	Since $ c+J $ is an element of $ \mathcal{M}_1 $, we can use Lemma \ref{lma-Alg-v1} to conclude that the element $ a $ is also a member of $ \mathcal{M}_1 $. Hence, by Theorem \ref{thm-NAlgModelcomplete}, the set of formulas (\ref{eq-Alg-main-lma}) is also satisfiable in $ \mathcal{M}_1 $.
\end{proof}

\vspace*{5pt}
\begin{mthm}
\label{thmMain}
Suppose that $\alpha$ is a  transcendental real number. Then, $ \Ta $ is a complete and model-complete axiomatization for the structure $ \Za $, which has the strict order property. Moreover, $ \Ta $ is decidable whenever $\alpha$ is computable. 
\end{mthm}
\begin{proof}
Based on Corollary \ref{crl-non-alg}, we can assume that the non-algebraic formulas $\Gamma(\bar{x};\bar{y})$ appearing in Lemma \ref{lma-Alg} are in the most general possible form. That is, we do not need to consider sets of non-algebraic formulas of the form $\Gamma\big(h(x_1,\ldots,x_n);\bar{y}\big)$, $\Gamma\big(h_1(x_1),\ldots,h_n(x_n);\bar{y}\big)$, or the like.

Also, using Lemma \ref{lma-Alg-diff-negative} and similar to the proof of Corollary \ref{crl-non-alg}, any term of the language can be replaced by an addition of single-variable terms at the expense of adding finitely many non-algebraic formulas.   Hence, the most general form of a quantifier-free $\mathcal{L}$-formula can be considered as the conjunction of a finite set of non-algebraic formulas $\Gamma(x_1,\ldots,x_n;\bar{b})$ with finitely many algebraic formulas of the form $h_1(x_1)+\ldots+ h_n(x_n)=z$.

Moreover, Lemma \ref{lma-Alg-technical} shows that the existence of one algebraic equation has two important effects: Firstly, it reduces the number of free variables $x_1,\ldots, x_n$ required to solve the formula in the smaller model. Secondly, it turns the whole formula into a disjunction of conjunctions of non-algebraic formulas, which of course contain an extra variable $w$. The proof of Lemma \ref{lma-Alg} shows how to handle this extra variable. Therefore, Lemma \ref{lma-Alg} suffices for model-completeness, and there is no need to enrich \eqref{eq-Alg-main-lma} with more algebraic formulas.

 As mentioned in part (1) of Remark \ref{rem-prime-computability-cut}, Axiom scheme \ref{ax-cut} guarantees that $\Za$ is the prime model of $\Ta$, and this proves the completeness of this theory. 

It is obvious that Axiom \ref{axBasicDLO} and Axiom scheme \ref{axKronecker} are recursively enumerable even if $\alpha$ is not computable. As addressed in Remark \ref{rem-computablity-non-alg}, the computability of $\alpha$ is not required for listing the instances of Axiom scheme \ref{axNonAlg} either. But, as mentioned in part (2) of Remark \ref{rem-prime-computability-cut} and also in Remark \ref{rem-computability-h}, the computability of $\alpha$ is needed and in fact suffices for Axiom schemes \ref{ax-cut}, \ref{ax-RangeOfTerms}, and \ref{ax-Bounds} to be recursively enumerable.

The strict order property is implied by the fact that the dense linear order relation ``$ \fp{\alpha x}<\fp{\alpha y} $'' is definable in $ \Ta $ .
\end{proof}
\vspace*{5pt}

\begin{rem}
By part (1) of Remark \ref{rem-prime-computability-cut}, a direct consequence of Axiom scheme \ref{ax-cut} is that the first-order theory of structures $\Za$ and $\Za[\beta]$ are not elementarily equivalent for different numbers $\alpha $ and $\beta$. In view of the results appearing in \cite{Hieronymi-SturmianWords}, part of the following common theory is reflected in our axioms:
 \[ \bigcap_{\alpha\in\QQc}\Ta.  \]
\end{rem}

\begin{rem}\label{rem-QE}
	Simpler forms of the techniques used in this paper are applied, among others, 
	in \cite{af-moh} to eliminate quantifiers for 
	$ \Ta $ when $ \alpha $ is the golden ratio. In this case, the specific formula that expresses being in the range of $ f $, namely $ \exists y f(y)=x $, is equivalent to the quantifier-free formula $ x=f(f(x)-x+1) $. 
	This is mainly due to the fact that for every  $ x\in \ZZ $ we have $f^2(x)=f(x)+x-1$, which in turn holds because of the algebraic dependence $\alpha^2=\alpha+1$.
 	There does not seem to exist an easy way to eliminate quantifiers even for such a simple formula in the case of a transcendental $ \alpha $.
\end{rem}

\begin{rem}\label{rem-EffectiveModelCompleteness}
	Our proofs in this paper can be checked to show also an \textit{effective} model-completeness result for the theory of $ \Za $. In fact, based on the proofs appeared in Section \ref{secNonAlg}, one can use the effective quantifier elimination available in the ordered field of reals to effectively obtain an equivalent quantifier-free formula for any formula of the form $ \exists\bar{x}\theta(\bar{x};\bar{y}) $ where $ \theta(\bar{x};\bar{y}) $ is non-algebraic. Using formula (\ref{eq-Alg-diff}) and Axiom \ref{ax-Bounds}, one can effectively find a universal formula equivalent to $ \exists x \big(h(x)=y\big) $ for an $f$-polynomial $ h(x) $. For an example, when $ \alpha $ equals Euler's number $ e $ the formula $ \exists x \big(f(x)=y\big) $ is equivalent to 
	\[ \forall x\Big(y-1\neq f(x)\wedge y+1\neq f(x)\Big). \]
	For systems involving more than one existential variables and 	containing an algebraic formula, one can apply the proof of Lemma \ref{lma-Alg-technical} to effectively reduce the number of existential variables.
\end{rem}

\section{Concluding remarks}\label{secConclude}

\subsection{The case of an algebraic $\alpha$}

The techniques used in this paper can be applied to obtain the same result for an algebraic $ \alpha $. In fact, when $ \alpha $ satisfies an equation of a minimal degree like
\begin{align}\label{eq-alphaAlg}
	\alpha^n+k_{n-1}\alpha^{n-1}+\ldots+k_0=0,
\end{align}
with integer coefficients, we can use (\ref{eq-alphaAlg}) to calculate the value of a decimal $ \fp{\alpha f^m(x)} $, with $ m\geq n $, in terms of the decimals
\[ \fp{\alpha f(x)}, \ldots, \fp{\alpha f^{n-1}(x)}. \]
At the same time, each  $f$-polynomial $ h(x) $ can be assumed to contain  powers of $ f $ strictly less than $ n $. Having made the mentioned adjustments, the rest of the argument can easily be carried out.

\subsection{On definable sets}
\label{subdefinable}
Based on the terminology used in \cite[Section 3.1]{Tao-StructureAndRandomness}, there appear three fundamental types of sets in various areas of mathematics: The ``structured'' sets, the ``random'' sets, and sets of a ``hybrid'' nature. Below, we make a concise clarification on this phenomena concerning the definable sets in $ \Za $. 

If a power of $ f $ does not appear in an existential formula $ \varphi(x) $ with a single free variable $ x $, then the quantifier elimination available in Presburger arithmetic shows that $ \varphi(x) $ is actually describing a congruence class to which $ x $ belongs. So in this case, $ \varphi(x) $ defines an infinite arithmetic progression which is a typical example of a ``structured'' set by having a completely predictable behaviour. 

On the contrary, Connell proved in \cite[Theorem 2]{Connell-BeattySeqsII} that no Beatty sequence with an irrational modulo can contain an infinite arithmetic progression. That is the set of solutions of a formula like $ \exists y \big(x=f^n(y)\big) $, that forms a typical example of an existential formula containing a power of $ f $, cannot contain an infinite arithmetic progression. It is not clear to us if the same fact holds for formulas of the form $ \exists y \big(x=f(y) + f^2(y)\big) $ that contain an addition of terms; the latter question might be of interest from the perspective of additive combinatorics.

However in proposition below, and using Theorem \ref{Krock}, we prove that in the range of any term of one variable we can find finite arithmetic progressions of arbitrary large length. This slightly generalizes a similar result by Connell in \cite[Theorem 2]{Connell-BeattySeqsII}.

\begin{prop}
	Let $ h(x)=\sum_{i=0}^k m_if^i(x) $. For every $ n\in \NN $ there exists an arithmetic progression of length $ n $ in the range of $ h $ in $ \Za $.	
\end{prop}
\begin{proof}
	To have an arithmetic progression of length $ n $, it suffices to find $ x $ and $ y $ such that for every $ \ell\leq n $ we have that $ h(x+\ell y)=h(x) + \ell h(y). $ And the latter holds whenever there are $ x,y\in \ZZ $ such that $ \Za $ satisfies the following non-algebraic formula for every $ 1\leq i\leq k $
	\[ f^i(x+ny)=f^i(x)+nf^i(y), \]
	or equivalently whenever we have the following for every $ 0\leq i\leq k-1 $ 
	\[ \Za\models 0<\fp{\alpha f^i(x)}+n \fp{\alpha f^i(y)}<1. \]
	But, Theorem \ref{Krock} allows us to find $ x $ and $ y $ with the desired properties.
\end{proof}
A similar argument as in the proof of proposition above shows that for an existential formula $ \varphi(x) $ of more than one existential variable, the set of solutions $ \varphi(\Za) $ contains arithmetic progressions of arbitrary finite lengths.

The latter observation shows that the formulas containing a power of $ f $ behaves more-or-less similar to prime numbers in that they do not contain infinite arithmetic progression whereas they do contain arbitrary long finite arithmetic progressions (\cite{GreenTao-primesArithProg}). However, Proposition \ref{prop-Idea} shows that such definable sets may differ from the primes in intersecting each congruence class at infinitely many points. But it seems reasonable to consider them as hybrid sets in $ \Za $ just like as we do for primes in $ \ZZ $. 

To sum up, the structured definable sets in $ \Za $ are disjoint-by-finite from the hybrid sets, and still another interesting phenomena occurs in $ \Za $ when we consider two mentioned types of sets from the perspective of the order topology available in $ \Za $ by the formula $ \fp{\alpha x} < \fp{\alpha y} $. In fact, both the structured and hybrid sets find a uniform description in this topology by being simultaneously dense and co-dense there.

\subsection{A connection to o-minimality}\label{subsecOMin}

We show that the non-algebraic part of $ \Ta $, which we have denoted by $ \mathcal{T}_{\nalg} $ in Section \ref{secNonAlg}, gives rise to an o-minimal theory that embodies its main features.

First for a model $ \mathcal{M}\models\mathcal{T}_{\nalg} $ we associate a structure $ \mathcal{A}_\mathcal{M} $ in a language $ \mathcal{L}^* $ that contains a set of predicates meant to capture the non-algebraic content of $ \mathcal{M} $. 


So let $ \mathcal{L}^*=\Big\{<, P_{\bar{m},\bar{n},\ell}\Big\}_{\bar{m},\bar{n},\ell\in\ZZ}$ where each $ P_{\bar{m},\bar{n},\ell} $ accepts tuples of arity $ |\bar{m}|+|\bar{n}| $. 
Fix $ \mathcal{M}\models \mathcal{T}_{\nalg} $ and let $ A_\mathcal{M}  $ be the subset of (possibly non-standard) reals defined as
\[ A_\mathcal{M}:=\Big\{\fp{\alpha a}\big|a\in M\Big\}. \]
%
For $ \fp{\alpha a_1},\ldots,\fp{\alpha b_1},\ldots\in A_\mathcal{M} $, we let $ P_{\bar{m},\bar{n},\ell}(\fp{\alpha a_1},\ldots,\fp{\alpha b_1},\ldots) $ hold in $ \mathcal{A}_\mathcal{M} $ if and only if

\begin{align}\label{eq-omin}
	\mathcal{M}\models \sum m_i\fp{\alpha a_i} <\sum n_i\fp{\alpha b_i}+\ell.
\end{align}

Note in particular that $ P_{\bar{1},\bar{1},0}(\fp{\alpha a},\fp{\alpha b}) $ holds in $ \mathcal{A}_\mathcal{M} $ if and only if 
\begin{align*}
	\mathcal{M}\models \fp{\alpha a}<\fp{\alpha b}.
\end{align*}
That is, $  P_{\bar{1},\bar{1},0} $ coincides with the relation $ < $ in $ \mathcal{A}_\mathcal{M} $. Hence by Axiom 1 this predicate defines a dense linear ordering on $ \mathcal{A}_\mathcal{M} $. 

Towards introducing $ T^* $, we keep using the notation $ \fp{\alpha x} $ for elements of an arbitrary $ \mathcal{L}^* $-structure $ \mathcal{A} $. Also, for simplicity and particularly in axiom schemes (2) and (3) below, we keep thinking of predicates $ P_{\bar{m},\bar{n},l} $ as if they are reflecting the content of the inequality appeared in \eqref{eq-omin}, while we carefully have this reservation in mind that an expression like $ \sum m_i\fp{\alpha a_i} $ is, by itself, just meaningless in $ T^* $ and does not refer to an actual point.

Let $T^*$ be the theory that
describes the following:
\begin{enumerate}

	\item[(1)] 
	The relation $<$ is a dense linear order.

	\item[(2)] 
	The predicates 
	$P_{\bar{m},\bar{n},\ell}$ 
	are consistent
	with 
	the usual addition and ordering of real numbers.
	That is $T^*$ describes how
	elements can be moved from 
	each side of \eqref{eq-omin}
	to the other.
	For example if $ P_{\overline{2},\bar{1},0}(a,c) $ holds, then we have that $ P_{\overline{1},\overline{1,-1},0}(a,c,a) $. This example reflects the content of the fact that $2[\alpha a]<[\alpha c]$ implies $[\alpha a]<[\alpha c]-[\alpha a]$ in real numbers.
	
	\item[(3)]
	If  $\sum m_i\fp{\alpha a_i} <\sum n_i\fp{\alpha b_i}+\ell<1$ then 
	there is $[\alpha x]$ such that
	\[
	\sum m_i\fp{\alpha a_i} <[\alpha x]<\sum n_i\fp{\alpha b_i}+\ell. 
	\]
\end{enumerate}

Because of the density enforced on the predicates of $ \mathcal{L}^* $ by the axioms (1) and (3) above, it is easy to verify the following proposition.

\begin{prop}
	$T^*$ admits quantifier elimination in $ \mathcal{L}^* $.
\end{prop}

Now, for some model $ \mathcal{M}\models\mathcal{T}_{\nalg} $ it is easy to see that the associated $ \mathcal{A}_\mathcal{M} $ is a model of $ T^* $. On the other hand, $T^*$ is \textit{similar to} an  o-minimal theory in the sense that
any set 
defined by a formula $\varphi(x,\bar{a},\bar{b})$
is a finite union of \textit{intervals}
of the form below
\[
\left\{x:\sum m_i\fp{\alpha a_i} <m\fp{\alpha x}<\sum n_i\fp{\alpha b_i}+\ell\right\}.
\]

But, as mentioned earlier, the endpoints of this \textit{interval} are not  actual points in an arbitrary model of $ T^* $. However, in each of the structures $ \mathcal{A}_\mathcal{M} $ these endpoints turn out to be elements of the form $ \fp{\alpha a} $. Moreover, at the expense of adding/subtracting an integer value to/from $ \ell $, we can write $ m\fp{\alpha x}$ as $ \fp{m\alpha x} $ or equivalently as $ \fp{\alpha z} $ for some $ z $ in $ \mathcal{M} $. That is, each $ \mathcal{L}^* $-formula  $\varphi(x,\bar{a},\bar{b})$ becomes equivalent to a finite disjunction of formulas of the form below in $ \mathcal{A}_\mathcal{M} $:
\[  \fp{\alpha a}< \fp{\alpha x}<\fp{\alpha b}. \]

Hence for any pair of models $ \mathcal{M}, \mathcal{N}\models\mathcal{T}_{\nalg} $ the two associated structures $ \mathcal{A}_{\mathcal{M}} $ and $ \mathcal{A}_{\mathcal{N}} $ are elementary equivalent since we are able to form a back-and-forth system between them. In other words, there exists a completion of $ T^* $ that is o-minimal and is determined by $ \mathcal{T}_{\nalg} $.

\vspace*{6pt}
We finish by posing the following question which is seemingly a natural continuation of the results appeared in this paper:

\vspace*{6pt}
\textbf{Question.} Is the structure $ \langle\ZZ,<,+,-,0,1,f\rangle $, which is, $ \Za $ augmented by the usual ordering of integers, decidable? Is it model-complete? Or, does it admit quantifier elimination in a naturally expanded language?

\subsection*{Acknowledgments}
We would like to thank Philip Hieronymi, for generously offering several questions 
to work on
(including this and a special case of it that the first and third author treated in their previous paper (\cite{af-moh}).
The second author was in part supported by a grant from IPM (No. 99030044). We are also thankful to the anonymous referee whose comments led to a much better version of this paper.


\end{document}